\documentclass[11pt]{amsart}
\usepackage{amssymb}
\usepackage{amsmath}
\usepackage{fancyhdr}
\usepackage[british]{babel}
\usepackage{geometry}

\title[]{Optimal path and cycle decompositions of dense quasirandom graphs}

\date{\today}
\author{Stefan Glock, Daniela K\"uhn and Deryk Osthus}
\thanks{The research leading to these results was partially supported by the European Research Council
under the European Union's Seventh Framework Programme (FP/2007--2013) / ERC Grant
Agreements no. 258345 (D.~K\"uhn) and 306349 (S.~Glock and D.~Osthus).}

\newtheorem{firstthm}{Proposition}[section]
\newtheorem{theorem}[firstthm]{Theorem}
\newtheorem{prop}[firstthm]{Proposition}
\newtheorem{lemma}[firstthm]{Lemma}
\newtheorem{cor}[firstthm]{Corollary}

\newtheorem{conj}[firstthm]{Conjecture}

\def\noproof{{\unskip\nobreak\hfill\penalty50\hskip2em\hbox{}\nobreak\hfill%
       $\square$\parfillskip=0pt\finalhyphendemerits=0\par}\goodbreak}
\def\endproof{\noproof\bigskip}
\newdimen\margin   % needed for macros \textdisplay & \ltextdisplay
\def\textno#1&#2\par{%
   \margin=\hsize
   \advance\margin by -4\parindent
          \setbox1=\hbox{\sl#1}%
   \ifdim\wd1 < \margin
      $$\box1\eqno#2$$%
   \else
      \bigbreak
      \hbox to \hsize{\indent$\vcenter{\advance\hsize by -3\parindent
      \it\noindent#1}\hfil#2$}%
      \bigbreak
   \fi}
\def\proof{\removelastskip\penalty55\medskip\noindent{\bf Proof. }}
\def\lateproof#1{\removelastskip\penalty55\medskip\noindent{\bf Proof of #1. }}

\begin{document}

\def\COMMENT#1{}
\def\TASK#1{}

\def\eps{{\varepsilon}}
\newcommand{\ex}{\mathbb{E}}
\newcommand{\pr}{\mathbb{P}}
\newcommand{\cB}{\mathcal{B}}
\newcommand{\cE}{\mathcal{E}}
\newcommand{\cS}{\mathcal{S}}
\newcommand{\cF}{\mathcal{F}}
\newcommand{\cC}{\mathcal{C}}
\newcommand{\bN}{\mathbb{N}}
\newcommand{\bR}{\mathbb{R}}
\def\O{\mathcal{O}}
\newcommand{\cP}{\mathcal{P}}
\newcommand{\cQ}{\mathcal{Q}}
\newcommand{\cR}{\mathcal{R}}
\newcommand{\cK}{\mathcal{K}}
\newcommand{\cD}{\mathcal{D}}
\newcommand{\cI}{\mathcal{I}}
\newcommand{\cV}{\mathcal{V}}
\newcommand{\1}{{\bf 1}_{n\not\equiv \delta}}
\newcommand{\eul}{{\rm e}}

% own
\newcommand{\prob}[1]{\mathrm{\mathbb{P}}(#1)}
\newcommand{\expn}{\mathrm{\mathbb{E}}}
\newcommand{\whp}{a.a.s. }
\newcommand{\Whp}{A.a.s. }
\def\gnp{G_{n,p}}
\def\G{\mathcal{G}}
\def\lflr{\left\lfloor}
\def\rflr{\right\rfloor}
\def\lcl{\left\lceil}
\def\rcl{\right\rceil}

\newcommand{\brackets}[1]{\left(#1\right)}
\def\sm{\setminus}
\newcommand{\Set}[1]{\{#1\}}
\newcommand{\set}[2]{\{#1\,:\;#2\}}
\def\In{\subseteq}
\def\all{\forall \,}

\begin{abstract}  \noindent
Motivated by longstanding conjectures regarding decompositions of graphs into paths and cycles, we prove the following optimal decomposition results for random graphs. Let $0<p<1$ be constant and let $G\sim\gnp$. Let $odd(G)$ be the number of odd degree vertices in $G$. Then \whp the following hold:
\begin{itemize}
\item [(i)] $G$ can be decomposed into $\lflr\Delta(G)/2\rflr$ cycles and a matching of size $odd(G)/2$.
\item [(ii)] $G$ can be decomposed into $\max\Set{odd(G)/2,\lceil\Delta(G)/2\rceil}$ paths.
\item [(iii)] $G$ can be decomposed into $\lceil\Delta(G)/2\rceil$ linear forests.
\end{itemize}
Each of these bounds is best possible.
We actually derive (i)--(iii) from `quasirandom' versions of our results. In that context, we also determine the edge chromatic number of a given dense quasirandom graph of even order. For all these results, our main tool is a result on Hamilton decompositions of robust expanders by K{\"u}hn and Osthus.
\end{abstract}

\maketitle

 %ONLINE abstract version

%Motivated by longstanding conjectures regarding decompositions of graphs into paths and cycles, we prove the following optimal decomposition results for random graphs. Let $0<p<1$ be constant and let $G\sim G_{n,p}$. Let $odd(G)$ be the number of odd degree vertices in $G$. Then a.a.s. the following hold:
%
%(i) $G$ can be decomposed into $\lfloor\Delta(G)/2\rfloor$ cycles and a matching of size $odd(G)/2$.
%
%(ii) $G$ can be decomposed into $\max\{odd(G)/2,\lceil\Delta(G)/2\rceil\}$ paths.
%
%(iii) $G$ can be decomposed into $\lceil\Delta(G)/2\rceil$ linear forests.
%
%Each of these bounds is best possible.
%We actually derive (i)--(iii) from `quasirandom' versions of our results. In that context, we also determine the edge chromatic number of a given dense quasirandom graph of even order. For all these results, our main tool is a result on Hamilton decompositions of robust expanders by K\"uhn and Osthus.
%
%Joint work with Daniela K\"uhn and Deryk Osthus

\section{Introduction}\label{sec:intro}

There are several longstanding and beautiful conjectures on decompositions of graphs into cycles and/or paths. In this paper, we consider four of the most well-known in the setting of dense quasirandom and random graphs: the Erd\H{o}s-Gallai conjecture, the Gallai conjecture on path decompositions, the linear arboricity conjecture as well as the overfull subgraph conjecture.

\subsection{Decompositions of random graphs}

A classical result of Lov\'asz \cite{L} on decompositions of graphs states that the edges of any graph on $n$ vertices can be decomposed into at most $\lfloor n/2\rfloor$ cycles and paths. Erd\H{o}s and Gallai \cite{E,EGP}\COMMENT{cited as in \cite{KKS}} made the related conjecture that the edges of every graph $G$ on $n$ vertices can be decomposed into $\O(n)$ cycles and edges.\COMMENT{often cited along with this is the result of Pyber that every graph on $n$ vertices can be covered by $n-1$ cycles, which was also conjectured by Erdos. He also proved that every connected graph can be covered by $n/2+\O(n^{3/4})$ paths and pointed out that if one proves an asymptotic version of the Gallai conjecture then one proves the conjecture.} Conlon, Fox and Sudakov \cite{CFS} recently showed that $\O(n\log \log n)$ cycles and edges suffice and that the conjecture holds for graphs with linear minimum degree. They also proved that the conjecture holds \whp for the binomial random graph $G\sim\gnp$. Kor{\'a}ndi, Krivelevich and Sudakov \cite{KKS} carried out a more systematic study of the problem for $\gnp$: for a large range of $p$, \whp $\gnp$ can be decomposed into $n/4+np/2+o(n)$ cycles and edges, which is asymptotically best possible. They also asked for improved error terms. For constant $p$, we will give an exact formula.

A further related conjecture of Gallai (see \cite{L})\COMMENT{Lovasz says that Erdos asked for the minimal number of paths needed and Gallai gave the number $n/2$.} states that every connected graph on $n$ vertices can be decomposed into $\lceil n/2\rceil$ paths. The result of Lov\'asz mentioned above implies that for every (not necessarily connected) graph $n-1$ paths suffice.\COMMENT{Lovasz deduced from his theorem that $odd(G)/2+even(G)-1$ paths suffice if $even(G)\geq1$ (by adding some vertices). If $even(G)=0$, then $n$ is even and so we have a decomposition into $n/2=odd(G)/2$ paths and cycles, which implies that there are no cycles. Since $n-1$ is not really better than $n$ (obvious by splitting every cycle), the effort to get the $-1$ is maybe not worth it. } This has been improved to $\lfloor2n/3\rfloor$ paths \cite{DK,Y}. Here we determine the number of paths in an optimal path decomposition of $\gnp$ for constant $p$. In particular this implies that Gallai's conjecture holds (with room to spare) for almost all graphs.

Next, recall that an edge colouring of a graph is a partition of its edge set into matchings. A matching can be viewed as a forest whose connected components are edges. As a relaxation of this, a \emph{linear forest} is a forest whose components are paths, and the least possible number of linear forests needed to partition the edge set of a graph $G$ is called the \emph{linear arboricity} of $G$, denoted by $la(G)$. Clearly, in order to cover all edges at any vertex of maximum degree, we need at least $\lceil\Delta(G)/2\rceil$ linear forests. However, for some graphs (e.g. complete graphs on an odd number of vertices) we need at least $\lceil(\Delta(G)+1)/2\rceil$ linear forests. The following conjecture is known as the linear arboricity conjecture and can be viewed as an analogue to Vizing's theorem.

\begin{conj}[Akiyama, Exoo, Harary \cite{AEH}] \label{conj:forests}
For every graph $G$, $la(G)\leq\lceil(\Delta(G)+1)/2\rceil$.
\end{conj}

This is equivalent to the statement that for all $d$-regular graphs $G$, $la(G)=\lceil(d+1)/2\rceil$.\COMMENT{Assume the conjecture is true and let $G$ be $d$-regular. By conjecture, $la(G)\leq\lceil(d+1)/2\rceil$. For $d$ odd, this is equal to the lower bound $\lceil d/2\rceil$. For $d$ even, the lower bound $\lceil d/2\rceil$ cannot be attained as every vertex must have been an inner vertex of every linear forest.
Suppose the statement on regular graphs holds. For a graph $G$ which is not regular, we need to find a graph $G'$ that is $\Delta(G)$-regular and contains $G$ as a subgraph. Then, each linear forest of $G'$ induces a linear forest of $G$. If $x$ does not have maximum degree, take $\Delta(G)-d_{G_i}(x)+1$ copies of $G_i$ (current graph) and add a complete graph on all copies of $x$. Proceed until all vertices have degree $\Delta(G)$.}
Alon \cite{A} proved an approximate version of the conjecture for sufficiently large values of $\Delta(G)$. Using his approach, McDiarmid and Reed \cite{MR} confirmed the conjecture for random regular graphs with fixed degree. We will show that, for a large range of $p$, \whp the random graph $G\sim\gnp$ can be decomposed into $\lceil\Delta(G)/2\rceil$ linear forests. Moreover, we use the recent confirmation \cite{CBKLOT} of the so-called `Hamilton decomposition conjecture' to deduce that the linear arboricity conjecture holds for large and sufficiently dense regular graphs (see Corollary~\ref{cor:forests large}).

The following theorem summarises our optimal decomposition results for dense random graphs. We denote by $odd(G)$ the number of odd degree vertices in a graph $G$.

\begin{theorem} \label{thm:gnp}
Let $0<p<1$ be constant and let $G\sim\gnp$. Then \whp the following hold:
\begin{itemize}
\item [(i)] $G$ can be decomposed into $\lflr\Delta(G)/2\rflr$ cycles and a matching of size $odd(G)/2$.
\item [(ii)] $G$ can be decomposed into $\max\Set{odd(G)/2,\lceil\Delta(G)/2\rceil}$ paths.
\item [(iii)] $G$ can be decomposed into $\lceil\Delta(G)/2\rceil$ linear forests, i.e. $la(G)=\lceil\Delta(G)/2\rceil$.
\end{itemize}
\end{theorem}

Clearly, each of the given bounds is best possible.\COMMENT{best possible cycle: Let $G$ be any graph and let $H$ be a spanning subgraph of $G$ such that $G\sm H$ can be decomposed into $c$ cycles. Let $e=|E(H)|$. We want to bound $e+c$ from below. Clearly, $c\geq\Delta(G\sm H)/2\geq(\Delta(G)-\Delta(H))/2$. Let $S$ be the set of vertices with odd degree in $G$. Since $G\sm H$ is Eulerian, we have $d_H(x)\geq1$ for all $x\in S$. Therefore, $2|E(H)|=\sum d_H(x)=|S|+\sum_{x\in S}(d_H(x)-1)+\sum_{x\not\in S}d_H(x)\geq|S|+\Delta(H)-1$. We conclude that $e+c\geq|S|/2+\Delta(G)/2-1/2$, implying $e+c\geq odd(G)/2+\lflr\Delta(G)/2\rflr$.} 
Moreover, as observed e.g. in \cite{KKS} for a large range of $p$, \whp $odd(\gnp)=(1+o(1))n/2$. This means that for fixed $p<1/2$, the size of an optimal path decomposition of $\gnp$ is determined by the number of odd degree vertices, whereas for $p>1/2$, the maximum degree is the crucial parameter.

A related result of Gao, P\'erez-Gim\'enez and Sato \cite{GPS} determines the arboricity and spanning tree packing number of $\gnp$. Optimal results on packing Hamilton cycles in $\gnp$ which together cover essentially the whole range of $p$ were proven in \cite{KKO,KS}.

One can extend Theorem~\ref{thm:gnp}(iii) to the range $\frac{\log^{117}n}{n}\leq p=o(1)$ by applying a recent result in \cite{HKLO} on covering $\gnp$ by Hamilton cycles (see Corollary~\ref{cor:forests random}). It would be interesting to obtain corresponding exact results also for (i) and (ii). In particular we believe that the following should hold.

\begin{conj} \label{conj:paths}
Suppose $p=o(1)$ and $\frac{pn}{\log n}\rightarrow\infty$. Then \whp $G\sim\gnp$ can be decomposed into $odd(G)/2$ paths. 
\end{conj}

By tracking the number of cycles in the decomposition constructed in \cite{KKS} and by splitting every such cycle into two paths, one immediately obtains an approximate version of Conjecture~\ref{conj:paths}. Note that this argument does not yield an approximate version of Theorem~\ref{thm:gnp}(ii) in the case when $p$ is constant.

\subsection{Dense quasirandom graphs}

As mentioned earlier, we will deduce Theorem~\ref{thm:gnp} from quasirandom versions of these results. As our notion of quasi\-random\-ness, we will consider the following one-sided version of $\eps$-regularity. Let $0<\eps,p<1$. A graph $G$ on $n$ vertices is called \emph{lower-$(p,\eps)$-regular} if we have $e_G(S,T)\geq(p-\eps)|S||T|$ for all disjoint $S,T\In V(G)$ with $|S|,|T|\geq\eps n$. 

The next theorem is a quasirandom version of Theorem~\ref{thm:gnp}(i). Similarly, we will also prove quasirandom versions of parts (ii) and (iii) (see Theorem~\ref{thm:pathforests}).

\begin{theorem} \label{thm:Eulerian plain}
For all $0<p<1$ there exist $\eps,\eta>0$ such that for sufficiently large $n$, the following holds: Suppose $G$ is a lower-$(p,\eps)$-regular graph on $n$ vertices. Moreover, assume that $\Delta(G)-\delta(G)\leq\eta n$ and that $G$ is Eulerian. Then $G$ can be decomposed into $\Delta(G)/2$ cycles.
\end{theorem}

This confirms the following conjecture of Haj\'os (see \cite{L}) for quasirandom graphs (with room to spare): Every Eulerian graph on $n$ vertices has a decomposition into $\lfloor n/2\rfloor$ cycles. (It is easy to see that this conjecture implies the Erd\H{o}s-Gallai conjecture.)\COMMENT{In fact, they are equivalent if one replaces the bound in Haj\'os' conjecture by $\O(n)$.}

We will also apply our approach to edge colourings of dense quasirandom graphs. Recall that in general it is NP-complete to decide whether a graph $G$ has chromatic index $\Delta(G)$ or $\Delta(G)+1$ (see for example \cite{Ho}). We will show that for dense quasirandom graphs of even order this decision problem can be solved in quadratic time without being trivial. 
For this, call a subgraph $H$ of $G$ \emph{overfull} if $e(H)>\Delta(G)\lfloor|H|/2\rfloor$. Clearly, if $G$ contains any overfull subgraph, then $\chi'(G)=\Delta(G)+1$. The following conjecture is known as the overfull subgraph conjecture and dates back to 1986\COMMENT{actually 1985, conference in Denmark}.

\begin{conj}[Chetwynd, Hilton \cite{CH}]
A graph $G$ on $n$ vertices with $\Delta(G)>n/3$ satisfies $\chi'(G)=\Delta(G)$ if and only if $G$ contains no overfull subgraph.
\end{conj}

This conjecture implies the $1$-factorization conjecture, that every regular graph of sufficiently high degree and even order can be decomposed into perfect matchings, which was recently proved for large graphs in \cite{CBKLOT}. Minimum degree conditions under which the overfull subgraph conjecture is true were first investigated in \cite{BHV,P}. (We refer to \cite{SSTF} for a more thorough discussion of the area.) We prove the overfull subgraph conjecture for quasirandom graphs of even order, even if the maximum degree is smaller than stated in the conjecture, as long as it is linear.

\begin{theorem} \label{thm:colouring}
For all $0<p<1$ there exist $\eps,\eta>0$ such that for sufficiently large $n$, the following holds: Suppose $G$ is a lower-$(p,\eps)$-regular graph on $n$ vertices and $n$ is even. Moreover, assume that $\Delta(G)-\delta(G)\leq\eta n$. Then $\chi'(G)=\Delta(G)$ if and only if $G$ contains no overfull subgraph. Further, there is a polynomial time algorithm which finds an optimal colouring.
\end{theorem}

At the first glance, the overfull subgraph criterion seems not very helpful in terms of time complexity, as it involves all subgraphs of $G$. (On the other hand, Niessen \cite{N} proved that in the case when $\Delta(G)\geq|G|/2$ there is a polynomial time algorithm which finds all overfull subgraphs.) Our proof of Theorem~\ref{thm:colouring} will actually yield a simple criterion whether $G$ is class $1$ or class $2$. Moreover, the proof is constructive, thus using appropriate running time statements for our tools, this yields a polynomial time algorithm which finds an optimal colouring.

The condition of $n$ being even is essential for our proof as we will colour Hamilton cycles with two colours each. It would be interesting to obtain a similar result for graphs of odd order.

\begin{conj} \label{conj:odd}
For every $0<p<1$ there exist $\eps,\eta>0$ and $n_0\in\bN$ such that the following holds. Whenever $G$ is a lower-$(p,\eps)$-regular graph on $n\geq n_0$ vertices, where $n$ is odd, and $\Delta(G)-\delta(G)\leq\eta n$, then $\chi'(G)=\Delta(G)$ if and only if $\sum_{x\in V(G)}(\Delta(G)-d_G(x))\geq\Delta(G)$.
\end{conj}

Note that the condition $\sum_{x\in V(G)}(\Delta(G)-d_G(x))\geq\Delta(G)$ in Conjecture~\ref{conj:odd}\COMMENT{that is, if $n$ is odd} is equivalent to the requirement that $G$ itself is not overfull. Also note that the corresponding question for $\gnp$ is easily solved if $p$ does not tend to $0$ or $1$ too quickly: It is well-known that in this case \whp $G\sim\gnp$ satisfies $\chi'(G)=\Delta(G)$, which follows from the fact that \whp $G$ has a unique vertex of maximum degree.\COMMENT{and Vizing's adjacency lemma}

This paper is organised as follows. In the next section, we will introduce the basic terminology. Section~\ref{sec:tools} discusses lower-$\eps$-regularity, robust expansion and collects some probabilistic results that we will use later on. It also introduces our main tool, namely that robust expanders of linear degree have a Hamilton decomposition, which was recently proven in \cite{KO:13}. In Section~\ref{sec:cycles} we will prove Theorem~\ref{thm:Eulerian plain}. Building on this, we will prove a quasirandom version of Theorem~\ref{thm:gnp}(ii) and (iii) in Section~\ref{sec:paths}, which altogether implies Theorem~\ref{thm:gnp}. 
Section~\ref{sec:forests} contains two additional results concerning linear arboricity.
The last section is devoted to edge colouring, i.e. the proof of Theorem~\ref{thm:colouring}.

\section{Notation} \label{sec:notation}

All graphs considered in this paper are finite and do not contain loops. A graph with parallel edges is referred to as a multigraph. 
Let $G$ be a graph, multigraph or digraph. As usual, we let $V(G)$ and $E(G)$ denote the vertex set and edge set, respectively. We further let $|G|:=|V(G)|$ and $e(G):=|E(G)|$. Given $U\In V(G)$, $G[U]$ denotes the subgraph of $G$ induced by $U$, and $G-U:=G[V(G)\sm U]$. If $F\In E(G)$, then let $G\sm F$ be obtained from $G$ by removing all edges in $F$, and if $H$ is a subgraph of $G$, then $G\sm H:=G\sm E(H)$. Moreover, if $G$ is a graph and $F$ a set of edges in the complement of $G$, then we let $G\cup F$ denote the graph obtained from $G$ by adding the edges in $F$.

Given a graph $G$, the degree of a vertex $v$ is denoted by $d_G(v)$, and the set of all neighbours of $v$ is denoted by $N_G(v)$. For a set $S\In V(G)$, we define $d_S(v):=|N_G(v)\cap S|$.
For $S,T\In V(G)$ disjoint, let $e_G(S,T)$ be the number of edges in $G$ between $S$ and $T$.

For a digraph $G$, let $N_G^+(v)$ and $N_G^-(v)$ denote the outneighbourhood and inneighbourhood of a vertex $v$, respectively. Furthermore, $d_G^+(v)$ is the outdegree of $v$ and $d_G^-(v)$ the indegree. We also set $d_G(v):=d_G^+(v)+d_G^-(v)$ and then define $\Delta(G):=\max_{v\in V(G)}d_G(v)$ and $\delta(G):=\min_{v\in V(G)}d_G(v)$.
A digraph $G$ is called \emph{Eulerian} if $d_G^+(v)=d_G^-(v)$ for every vertex $v$, and \emph{$r$-regular} if $d_G^+(v)=d_G^-(v)=r$ for every vertex $v$. The \emph{minimum semidegree} of $G$ is defined as $\delta^0(G):=\min\Set{\delta^+(G),\delta^-(G)}$, where $\delta^\pm(G):=\min_{v\in V(G)}d_G^{\pm}(v)$.
For $S,T\In V(G)$ disjoint, let $e_G(S,T)$ be the number of edges in $G$ directed from $S$ to $T$.
Given a fixed digraph $G$, for a vertex $v$ and a set $S\In V(G)$, we define $d_S^{\pm}(v):=|N_G^{\pm}(v)\cap S|$. Paths and cycles in digraphs are always supposed to be directed.

If $G$ is a multigraph and $v\in V(G)$, then $d_G(v)$ counts the number of edges $v$ is incident with, where parallel edges are counted with multiplicity.

We denote by $\gnp$ the binomial random graph on $n$ vertices, that is, the random graph which is generated by including each of the $\binom{n}{2}$ possible edges independently with probability $p$. If a random variable $X$ has binomial distribution with parameters $n$ and $p$, we write $X\sim Bin(n,p)$.
We say that a property $P=P(n)$ holds \emph{asymptotically almost surely (a.a.s.)} if the probability that $P$ holds tends to $1$ as $n$ tends to infinity.
For the sake of readability, we assume large quantities to be integers whenever this does not affect the argument.

We will also make use of the following notation: $0<a \ll b\le 1$. More precisely, if we claim that a statement is true whenever $0<a \ll b\le 1$, then this means that there exists a non-decreasing function $f\colon (0,1] \rightarrow (0,1]$ such that the statement holds for all $0<a,b\le 1$ satisfying $a\le f(b)$.
Intuitively, this means that for every $b>0$, the statement is true provided that $a$ is sufficiently small compared to $b$.

\section{Tools and preliminary results} \label{sec:tools}

\subsection{Quasirandomness}

Here we collect some basic properties of lower-$(p,\eps)$-regular graphs and digraphs. Let $0<\eps,p<1$. A (di-)graph $G$ on $n$ vertices is called \emph{lower-$(p,\eps)$-regular} if we have $$e_G(S,T)\geq(p-\eps)|S||T|$$ for all disjoint $S,T\In V(G)$ with $|S|,|T|\geq\eps n$. 
Note that $p$ cannot be viewed as the density of $G$, as every spanning supergraph of a lower-$(p,\eps)$-regular graph is also lower-$(p,\eps)$-regular.

The following proposition says that we can modify a lower-$\eps$-regular graph slightly such that it is still lower-$\eps'$-regular. Its proof is straightforward using the definition of lower-$\eps$-regularity. 

\begin{prop} \label{prop:changes}

Let $0<1/n_0\ll\eps,p<1$. Let $G$ be a lower-$(p,\eps)$-regular (di-)graph on $n\geq n_0$ vertices. Then the following hold:

\begin{itemize}

\item [(i)] If $G'$ is obtained from $G$ by adding a new vertex $w$ and arbitrary edges at $w$, then $G'$ is lower-$(p,2\eps)$-regular. \COMMENT{Let $|S|,|T|\geq2\eps n$. Let $S':=S-w$ and $T':=T-w$. $$e_{G'}(S,T)\geq e_{G}(S',T')\geq (p-\eps)|S'||T'|\geq(p-2\eps)|S||T|.$$}

\item [(ii)] Let $H$ be a graph on $V(G)$ such that $\Delta(H)\leq\eta n$. Let $\eps':=\max\Set{2\eps,2\sqrt{\eta}}$. Then $G\sm H$ is lower-$(p,\eps')$-regular.\COMMENT{For disjoint sets $S,T$ with $|S|,|T|\geq\eps'n$, we have $e_{G'}(S,T)\geq e_G(S,T)-|S|\eta n\geq (p-\eps)|S||T|-\sqrt{\eta}|S||T|\geq(p-\eps')|S||T|.$}

\item [(iii)] If $U\In V(G)$ has size at least $\beta n$, then $G[U]$ is lower-$(p,\eps/\beta)$-regular.

\end{itemize}
\end{prop}

\begin{prop} \label{prop:minsemideg}
Let $0<1/n_0\ll\eps,\eta \ll p<1$. Let $G$ be a lower-$(p,\eps)$-regular digraph on $n\geq n_0$ vertices. Assume that $G$ is Eulerian and $\Delta(G)-\delta(G)\le \eta n$. Then $\delta^0(G)\ge pn/2$.
\end{prop}

\proof
Consider a partition of $V(G)$ into sets $S,T$ such that $|S|=\eps n$ and $|T|=(1-\eps)n$. We have $e_G(S,T)\geq(p-\eps)|S||T|$ and $e_G(S,T)\leq|S|\Delta(G)/2$. Thus, $\Delta(G)\geq2(p-\eps)(1-\eps)n$, implying $\delta^0(G)=\delta(G)/2\geq(p-\eps)(1-\eps)n-\eta n/2\geq pn/2$.
\endproof

The next lemma states that in a lower-$(p,\eps)$-regular digraph with linear minimum semidegree, we can always find a Hamilton cycle. This is a folklore observation for $\eps$-regularity, but in fact, carries over to lower-$\eps$-regular graphs and digraphs. It can be seen either by using the fact that every lower-$\eps$-regular digraph is a robust outexpander (see Subsection~\ref{subsec:robexp}) and then applying a result from \cite{KOT}, or directly by considering a longest directed path and using the lower-$\eps$-regularity to find a cycle on the same vertex set, which must then be a Hamilton cycle. Using the latter approach, it is moreover easy to see that a Hamilton cycle can be found in polynomial time: Start with any path and extend it in both directions until all inneighbours of the startvertex and all outneighbours of the endvertex lie on the path. Then use the lower-$\eps$-regularity to find a cycle on the same vertex set. If this is not a Hamilton cycle, modify the cycle into a path on more vertices and proceed as above.

\begin{lemma} \label{lem:hamcycle}
Let $0<1/n_0\ll\eps\ll\alpha,p<1$.
\begin{itemize}
\item[(i)] Let $G$ be a lower-$(p,\eps)$-regular digraph on $n\geq n_0$ vertices such that $\delta^0(G)\geq\alpha n$. Then $G$ contains a Hamilton cycle.
\item[(ii)] Let $G$ be a lower-$(p,\eps)$-regular graph on $n\geq n_0$ vertices such that $\delta(G)\geq\alpha n$. Then $G$ contains a Hamilton cycle.
\item[(iii)] Let $G$ be a lower-$(p,\eps)$-regular graph on $n\geq n_0$ vertices such that $\delta(G)\geq\alpha n$. Then for all distinct $a,b\in V(G)$, there exists a Hamilton path joining $a$ and~$b$.
\end{itemize}
Moreover, the Hamilton cycles in (i)--(ii) and the Hamilton path in (iii) can be found in polynomial time.
\end{lemma}

Clearly, (i) immediately implies (ii). To deduce (iii) from (i), contract $\Set{a,b}$ into a single vertex and orient the edges at $\Set{a,b}$ from $\Set{a,b}$ towards $N_G(a)\sm\Set{b}$ and from $N_G(b)\sm\Set{a}$ towards $\Set{a,b}$. Replace all other edges by two edges which are oriented in opposite directions. By (i) and Proposition~\ref{prop:changes}, the graph thus obtained contains a Hamilton cycle, which in turn induces a Hamilton path from $a$ to $b$ in $G$.\COMMENT{Proposition~\ref{prop:changes}(iii) implies that the deletion of $\Set{a,b}$ does not destroy regularity, inserting double edges makes a regular digraph, and adding vertex $\Set{a,b}$ with appropriate edges to $N_G(a)\sm\Set{b}$ and $N_G(b)\sm\Set{a}$ is ok by Proposition~\ref{prop:changes}(i).}

\subsection{Probabilistic preliminaries}

In this subsection, the following Chernoff-Hoeffding type bound is often used, sometimes without stating the explicit calculation.

\begin{lemma}[see \cite{JLR}] \label{lem:chernoff}
Suppose that $X$ is the sum of finitely many independent indicator variables. Then, for all $0<\eps<1$,
\begin{itemize}
\item[(i)] $\prob{X\leq(1-\eps)\expn{X}}\leq\eul^{-\frac{\eps^2}{3}\expn{X}},$
\item[(ii)] $\prob{X\geq(1+\eps)\expn{X}}\leq\eul^{-\frac{\eps^2}{3}\expn{X}}.$
\end{itemize}
In particular, the above bounds hold if $X$ has binomial distribution.
\end{lemma}

The next lemma collects some well-known properties of random graphs.

\begin{lemma} \label{lem:gnpbasics}
Let $0<\eps,p<1$ be constant. The following holds \whp for the random graph $G\sim\gnp$:
\begin{itemize}
\item[(i)] $\Delta(G)-\delta(G)\leq4\sqrt{n\log n}$,
\item[(ii)] $G$ is lower-$(p,\eps)$-regular,
\item[(iii)] $G$ has a unique vertex of maximum degree.
\end{itemize}
\end{lemma}

Indeed, using Lemma~\ref{lem:chernoff}, it is easy to see that \whp $|d_G(x)-np|\leq2\sqrt{n\log n}$ for all $x\in V(G)$, which implies (i). Similarly, using Lemma~\ref{lem:chernoff}, it is also straightforward to show (ii).\COMMENT{We can assume that $\eps\leq p$. Let $S,T\In V(G)$ be disjoint with $s:=|S|\geq\eps n$ and $t:=|T|\geq\eps n$. Then, by using Lemma~\ref{lem:chernoff},
\begin{align*}
\prob{|e_G(S,T)-pst|>\eps st}&\leq2\eul^{-(\eps/p)^2stp/3}\leq2\eul^{-2n},
\end{align*}
and a union bound over at most $4^n$ possibilities for $S,T$ completes the proof.}
For (iii), we refer to Theorem 3.15 in \cite{B}.

The following result is a variant of the observation that a random partition of a graph splits all vertex degrees evenly. Given a (di-)graph $G$ and a matching $M$ in the complete graph on $V(G)$, we say that a set $S\In V(G)$ \emph{separates} $M$ if there exists an edge $e$ of $M$ such that $e$ has precisely one endvertex in $S$.

\begin{lemma} \label{lem:auxsets}
Let $0<1/n_0\ll\alpha<1$. Suppose that $G$ is a digraph on $n\geq n_0$ vertices with $\delta^0(G)\geq\alpha n$. Let $M$ be any matching in the complete graph on $V(G)$. Then there exists a set of vertices $S$ for which the following hold:
\begin{itemize}
\item[(i)] $n/3\leq|S|\leq 2n/3$.
\item[(ii)] Every $x\in V(G)$ satisfies $d_S^+(x),d_S^-(x),d_{\overline{S}}^+(x),d_{\overline{S}}^-(x)\geq\alpha n/6$, where $\overline{S}:=V(G)\sm S$.
\item[(iii)] $S$ does not separate $M$.
\end{itemize}
\end{lemma}

\proof
Write $M=\Set{x_1y_1,\dots,x_my_m}$. By arbitrarily enlarging $M$ if necessary, we may assume that $M$ covers all but at most one vertex, i.e. $m=\lfloor n/2\rfloor$. Let $X:=\Set{x_1,\dots,x_m}$ and $Y:=\Set{y_1,\dots,y_m}$. Let $S$ be a random vertex set obtained by including each $\Set{x_i,y_i}$ independently with probability $1/2$, for all $i\in[m]$. We now show that $S$ satisfies (i)--(iii) simultaneously with positive probability. 
%We choose $S$ randomly and show that $S$ has the desired properties with probability greater than $0$. Therefore, independently for all $i\in\Set{1,\dots,m}$, include $\Set{x_i,y_i}$ in $S$ with probability $1/2$.
$S$ automatically satisfies (iii). Let $S_X:=S\cap X$. Then $|S_X|\sim Bin(m,1/2)$. Hence by Lemma~\ref{lem:chernoff}, $$\prob{||S_X|-m/2|>m/8}\leq2\eul^{-(1/4)^2m/6}.$$ Since $|S|=2|S_X|$, this implies that \whp $n/3\leq|S|\leq 2n/3$.

Consider any vertex $v\in V(G)$. 
Assume that $d^+_X(v)\geq(d_G^+(v)-1)/2$. Otherwise, $d^+_Y(v)\geq(d_G^+(v)-1)/2$ and we can proceed analogously. We have $d_{S_X}^+(v)\sim Bin(d^+_X(v),1/2)$. Thus, $$\prob{|d_{S_X}^+(v)-d^+_X(v)/2|>d^+_X(v)/8}\leq2\eul^{-(1/4)^2d^+_X(v)/6}\leq2\eul^{-\alpha n/300}\leq1/n^2.$$
Note that $3d^+_X(v)/8\leq d_{S_X}^+(v)\leq5d^+_X(v)/8$ and our assumption that $d_X^+(v)\ge (d_G^+(v)-1)/2$ together imply that $d_S^+(v)\geq d_{S_X}^+(v)\geq d_G^+(v)/6$ and $d_S^+(v)\leq d_{S_X}^+(v)+d^+_Y(v) \leq 5d^+_X(v)/8 + (d_G^+(v)-d^+_X(v))  \le 5d_G^+(v)/6$. We conclude $$\prob{|d_S^+(v)-d_G^+(v)/2|>d_G^+(v)/3}\leq1/n^2.$$ Taking a union bound, we see that a.a.s. $d_G^+(v)/6\leq d_S^+(v)\leq5d_G^+(v)/6$ for all vertices $v\in V(G)$. The same applies to $d_S^-(v)$.
\endproof

In order to obtain our optimal decomposition of $\gnp$ into cycles and edges, we want to find a perfect matching on the odd degree vertices, which is achieved by the following lemma.

\begin{lemma} \label{lem:odd matching}
Let $0<p<1$ be constant and $G\sim\gnp$. Then \whp $G[D]$ has a perfect matching, where $D$ is the set of odd degree vertices in $G$.
\end{lemma}

\proof
Let $\alpha:=p\cdot p_{min}/4$, where $p_{min}:=\min\Set{p,1-p}$. Choose a constant $\eps$ such that $0<\eps\ll\alpha,p$. First, we prove that \whp $\delta(G[D])\geq\alpha n$.

For any vertex $v$, let $d^{odd}_G(v)$ denote the number of neighbours of $v$ which have odd degree in $G$. Fix a vertex $v$ and choose any vertex $w$ distinct from $v$. Let $U:=V(G)\sm\Set{w}$ and let $\Omega$ contain all possible outcomes of $E(G[U])$. For $\tilde{E}\in\Omega$ and any vertex $x\in U$, let $d_{\tilde{E}}(x)$ denote the degree of $x$ in the graph $(U,\tilde{E})$. Let $\Omega_1:=\set{\tilde{E}\in\Omega}{d_{\tilde{E}}(v)\geq np/2}$ and $\Omega_2:=\Omega\sm\Omega_1$. We want to show that $\prob{d^{odd}_G(v)<\alpha n}\leq1/n^2$.\COMMENT{\begin{align*}
\prob{d^{odd}_G(v)<\alpha n}&=\sum_{\tilde{E}\in\Omega_1}\prob{d^{odd}_G(v)<\alpha n\mid E(G[U])=\tilde{E}}\cdot\prob{E(G[U])=\tilde{E}} \\
                            &+\sum_{\tilde{E}\in\Omega_2}\prob{d^{odd}_G(v)<\alpha n\mid E(G[U])=\tilde{E}}\cdot\prob{E(G[U])=\tilde{E}} \\
														&\leq\sum_{\tilde{E}\in\Omega_1}\prob{d^{odd}_G(v)<\alpha n\mid E(G[U])=\tilde{E}}\cdot\prob{E(G[U])=\tilde{E}} \\
                            &+\prob{E(G[U])\in\Omega_2}
\end{align*}}
Fix some $\tilde{E}\in\Omega_1$ and let $d:=d_{\tilde{E}}(v)$.\COMMENT{By fixing $\tilde{E}\in\Omega_1$, I don't mean to fix the outcome of $E(G[U])$. Just fix a subset of the edge set of the complete graph on $U$.} Let $v_1,\dots,v_d$ be the neighbours of $v$ in $(U,\tilde{E})$. For every $v_i$, let $X_i$ be the random indicator variable such that $X_i=1$ if $d_{\tilde{E}}(v_i)+e_G(\Set{v_i},\Set{w})$ is odd and $X_i=0$ otherwise. Let $X:=X_1+\dots+X_d$ and observe that $X_i$ depends only on the presence of the edge $v_iw$. Hence, $X_1,\dots,X_d$ are independent.
Observe that $\expn(X\mid E(G[U])=\tilde{E})\geq p_{min}d\geq2\alpha n$.\COMMENT{$\expn X=\sum_{i=1}^d\expn X_i$, $\expn X_i=\prob{X_i=1}=p$ or $1-p$.} Therefore, by using Lemma~\ref{lem:chernoff},
$$\prob{d^{odd}_G(v)<\alpha n\mid E(G[U])=\tilde{E}}\leq\COMMENT{not equal as $w$ might also be an odd neighbour of $v$}\prob{X<\alpha n\mid E(G[U])=\tilde{E}}\leq\eul^{-(1/2)^2\alpha n/3}.$$
Moreover, since $d_{G-w}(v)\sim Bin(n-2,p)$, we get $$\prob{E(G[U])\in\Omega_2}=\prob{d_{G-w}(v)<np/2}\leq\prob{d_{G-w}(v)<2(n-2)p/3}\leq\eul^{-(1/3)^2(n-2)p/3}.$$ Therefore, we can conclude that $$\prob{d^{odd}_G(v)<\alpha n}\leq\eul^{-\alpha n/12}+\eul^{-np/28}\leq1/n^2,$$ and a union bound yields that 
\whp $d^{odd}_G(v)\geq\alpha n$ for all $v\in V(G)$.

Further, Lemma~\ref{lem:gnpbasics} implies that \whp $G$ is lower-$(p,\eps)$-regular.
So we may assume that $G$ is lower-$(p,\eps)$-regular and that $\delta(G[D])\geq\alpha n$.
Since $|D|\geq\alpha n$, Proposition~\ref{prop:changes}(iii) now implies that $G[D]$ is lower-$(p,\eps/\alpha)$-regular. Thus, $G[D]$ has a perfect matching by Lemma~\ref{lem:hamcycle}(ii).
\endproof

\subsection{Robust expansion} \label{subsec:robexp}

Robust expansion is a natural relaxation of the concept of quasirandomness or $\eps$-regularity. Roughly speaking, a graph is a robust outexpander if every vertex set $S$ which is not too small and not too large has an outneighbourhood that is at least a little larger than $S$ even if we delete a small number of vertices or edges. More precisely, let $G$ be a digraph on $n$ vertices and let $0<\nu<1$. For a set $S\In V(G)$, the \emph{$\nu$-robust outneighbourhood} $RN^+_{\nu,G}(S)$ of $S$ is the set of all those vertices $x\in V(G)$ which have at least $\nu n$ inneighbours in $S$. If $\nu\leq\tau<1$ and $|RN^+_{\nu,G}(S)|\geq|S|+\nu n$ for all $S\In V(G)$ with $\tau n\leq|S|\leq(1-\tau)n$, then $G$ is called a \emph{robust $(\nu,\tau)$-outexpander}. The main result of \cite{KO:13} states that every regular robust outexpander of linear degree has a Hamilton decomposition, that is, a decomposition of its edges into edge-disjoint Hamilton cycles.

\begin{theorem}[K{\"u}hn, Osthus \cite{KO:13}] \label{thm:hamdecrobexp}
For all $\alpha>0$ there exists $\tau>0$ such that for every $\nu>0$ there is an integer $n_0=n_0(\alpha,\tau,\nu)$ for which the following holds. If $G$ is a robust $(\nu,\tau)$-outexpander on $n\geq n_0$ vertices and $G$ is an $r$-regular digraph, where $r\geq\alpha n$, then $G$ has a Hamilton decomposition.
\end{theorem}

Note that the function $\tau=\tau(\alpha)$ in Theorem~\ref{thm:hamdecrobexp} can be taken to be non-decreasing. Let $G$ be a graph on $n$ vertices and let $0<\nu<1$. For a set $S\In V(G)$, the \emph{$\nu$-robust neighbourhood} $RN_{\nu,G}(S)$ of $S$ is the set of all those vertices $x\in V(G)$ which have at least $\nu n$ neighbours in $S$. If $\nu\leq\tau<1$ and $|RN_{\nu,G}(S)|\geq|S|+\nu n$ for all $S\In V(G)$ with $\tau n\leq|S|\leq(1-\tau)n$, then $G$ is called a \emph{robust $(\nu,\tau)$-expander}. 
The following undirected version of Theorem~\ref{thm:hamdecrobexp} is deduced from Theorem~\ref{thm:hamdecrobexp} in \cite{KO:14}.

\begin{theorem}[K{\"u}hn, Osthus \cite{KO:14}] \label{thm:hamdecrobexpundirected}
For all $\alpha>0$ there exists $\tau>0$ such that for every $\nu>0$ there is an integer $n_0=n_0(\alpha,\tau,\nu)$ for which the following holds. If $G$ is a robust $(\nu,\tau)$-expander on $n\geq n_0$ vertices and $G$ is an $r$-regular graph where $r\geq\alpha n$ and $r$ is even, then $G$ has a Hamilton decomposition. Moreover, this decomposition can be found in time polynomial in $n$.
\end{theorem}

It is easy to see that lower-$\eps$-regularity implies robust expansion, so we can apply the above results to lower-$\eps$-regular graphs.

\begin{prop} \label{prop:REG2EXP}
Let $0<1/n_0\ll\eps\ll\tau,p<1$ and suppose that $G$ is a lower-$(p,\eps)$-regular (di-)graph on $n\geq n_0$ vertices. Then $G$ is also a robust $(\eps,\tau)$-(out)expander.\COMMENT
{Let $S\In V(G)$ with $\tau n\leq|S|\leq(1-\tau)n$. Let $RN:=RN^+_{\nu,G}(S)$ and suppose that $|RN|<|S|+\nu n$ ($\nu=\eps$). Then, $|\overline{RN}|\geq(\tau-\nu)n\geq\eps n$. Let $T$ be a subset of $\overline{RN}$ of size $\eps n$. Let $S'$ be a subset of $S$ of size at least $(\tau-\eps)n$ disjoint from $T$. Since $G$ is lower-$(p,\eps)$-regular, we have $e_G(S',T)\geq(p-\eps)(\tau-\eps) n|T|$. But since $T\In\overline{RN}$, we also have $e_G(S',T)<\nu n|T|$, thus implying $\nu\geq(p-\eps)(\tau-\eps)$, which is a contradiction if $\nu$ was chosen small enough.}
\end{prop}

Our next aim is to prove that an Eulerian lower-$\eps$-regular graph can be oriented in such a way that the digraph thus obtained is Eulerian and lower-$\eps$-regular. To this end, we need the following lemma from \cite{KO:13} involving subgraphs of robust outexpanders, which can be proven by means of the Max-Flow-Min-Cut theorem.

\begin{lemma} \label{lem:flow} 
Suppose that $0<1/n_0\ll\gamma,\xi\ll\nu\leq\tau\ll\alpha<1$. Let $G$ be a digraph on $n\geq n_0$ vertices with $\delta^0(G)\geq\alpha n$ which is a robust $(\nu,\tau)$-outexpander. For every vertex $x\in V(G)$, let $n^+_x,n^-_x\in\bN$ be such that $(1-\gamma)\xi n\leq n^+_x,n^-_x\leq(1+\gamma)\xi n$ and such that $\sum_{x\in V(G)}n^+_x=\sum_{x\in V(G)}n^-_x$. Then $G$ contains a spanning subdigraph $G'$ which satisfies $d_{G'}^+(x)=n^+_x$ and $d_{G'}^-(x)=n^-_x$ for every $x\in V(G)$.
\end{lemma}

\begin{lemma} \label{lem:orient}
Let $0<1/n_0\ll\eps\ll p,\alpha<1$. Suppose that $G$ is an Eulerian graph on $n\geq n_0$ vertices. Assume further that $G$ is lower-$(p,\eps)$-regular and that $\delta(G)\geq\alpha n$. Then there exists an orientation $G'$ of $G$ such that $G'$ is Eulerian and lower-$(p/4,\eps)$-regular.\COMMENT{$\delta^0(G')\geq\alpha n/2$ follows automatically}
\end{lemma}

\proof
Choose new constants $\gamma,\xi,\tau$ such that $0<1/n_0\ll\gamma,\xi\ll\eps\ll\tau\ll p,\alpha<1$. Observe that there exists a decomposition of the edges of $G$ into two digraphs $G_1,G_2$ such that the following holds for $i=1,2$.
\COMMENT{
(a) $d^+_{G_i}(x)\sim Bin(d_G(x),1/4)$. $$\prob{|d^+_{G_i}(x)-d_G(x)/4|>2\sqrt{n\log n}}\leq2\eul^{-(\frac{8\sqrt{n\log n}}{d_G(x)})^2d_G(x)/12}\leq2\eul^{-2\log n}\leq2/n^2$$
\newline
(b) Let $S,T$ be disjoint with $|S|,|T|\geq\eps n$. $e_{G_i}(S,T)\sim Bin(e_G(S,T),1/4)$. $$\prob{e_{G_i}(S,T)\leq(1-\eps)e_G(S,T)/4}\leq\eul^{-\eps^2e_G(S,T)/12}\leq\eul^{-\eps^4(p-\eps)n^2/12}\leq\eul^{-2n}$$ Thus \whp 
\begin{align*}
e_{G_i}(S,T)&\geq (1-\eps)e_G(S,T)/4\geq(1-\eps)(p-\eps)|S||T|/4 \\
            &=(p/4-(\eps+\eps p-\eps^2)/4)|S||T|\geq(p/4-\eps)|S||T|
\end{align*}}

\begin{itemize}
\item[(a)] $|d^\pm_{G_i}(x)-d_G(x)/4|\leq\gamma\xi n/2$ for all $x\in V(G)$.
\item[(b)] $G_i$ is lower-$(p/4,\eps)$-regular.
\end{itemize}
This can be seen as follows. If $e=xy\in E(G)$, orient $e$ from $x$ to $y$ with probability $1/2$ and from $y$ to $x$ otherwise and include $e$ in $G_1$ with probability $1/2$ and in $G_2$ otherwise. Do this independently for each edge and then use Lemma~\ref{lem:chernoff}.

In particular, we have $|d^+_{G_1}(x)-d^-_{G_1}(x)|\leq\gamma\xi n$ for all $x\in V(G)$ and $\delta^0(G_2)\geq\alpha n/5$. Moreover, (b) and Proposition~\ref{prop:REG2EXP} together imply that $G_2$ is a robust $(\eps,\tau)$-outexpander.\COMMENT{Here we need $\eps\ll\tau$}

Now, for every vertex $x$, let $n^+_x:=\xi n$ and $n^-_x:=d_{G_1}^+(x)-d_{G_1}^-(x)+\xi n$. Thus, $n^+_x,n^-_x=(1\pm\gamma)\xi n$ and $\sum_{x\in V(G)}n^+_x=\sum_{x\in V(G)}n^-_x$. Apply Lemma~\ref{lem:flow} (with $\eps$ playing the role of $\nu$) to $G_2$ to find a spanning subdigraph $G_2'$ of $G_2$ such that $d_{G_2'}^+(x)=n^+_x$ and $d_{G_2'}^-(x)=n^-_x$ for all $x\in V(G)$. Then, $G_1\cup G_2'$ is Eulerian. Let $H$ be the spanning subgraph of $G$ containing all those edges which are not contained in $G_1\cup G_2'$. Thus, $H$ is Eulerian and can be decomposed into cycles. Orient each cycle consistently to obtain an Eulerian orientation $H'$ of $H$. Then, $G':=G_1\cup G_2'\cup H'$ is the desired orientation of $G$. Indeed, $G'$ is lower-$(p/4,\eps)$-regular since $G_1$ is. 
\endproof

\section{Cycle decompositions} \label{sec:cycles}

We will now prove Theorem~\ref{thm:Eulerian plain}. In fact, we prove a directed version (Theorem~\ref{thm:Eulerian}) which also gives additional control on the decomposition. We then derive the undirected version (and thus Theorem~\ref{thm:Eulerian plain}) from this directed version. The additional control as well as the directed version will be of use when considering linear forests in Section~\ref{sec:paths}.

%In order to state Theorem~\ref{thm:Eulerian}, we need one more definition. Let $G$ be a (di-)graph and let $M=\Set{x_1y_1,\dots,x_my_m}$ be a matching in the complete graph on $V(G)$. We say that a sub(-di-)graph $F$ of $G$ is \emph{consistent with $M$}, if for all $i\in\Set{1,\dots,m}$, $x_i\in V(F)$ implies $y_i\in V(F)$ whenever $d_G(x_i)\le d_G(y_i)$.

In order to state Theorem~\ref{thm:Eulerian}, we need one more definition. Let $G$ be a (di-)graph and let $M=\Set{x_1y_1,\dots,x_my_m}$ be a matching in the complete graph on $V(G)$ such that $d_G(x_i)\leq d_G(y_i)$ for all $i\in\Set{1,\dots,m}$. We then say that a sub(-di-)graph $F$ of $G$ is \emph{consistent with $M$}, if for all $i\in\Set{1,\dots,m}$, $x_i\in V(F)$ implies $y_i\in V(F)$. Note that $F$ can be consistent with $M$ and separate $M$ at the same time, as $F$ might contain $y_i$ but not~$x_i$.

\begin{theorem} \label{thm:Eulerian}
Let $0<1/n_0\ll\eta,\eps\ll p<1$ be such that $\eps^2\le \eta$. Suppose $G$ is a lower-$(p,\eps)$-regular digraph on $n\geq n_0$ vertices. Moreover, assume that $G$ is Eulerian and that $\Delta(G)-\delta(G)\leq\eta n$. Let $M$ be any matching in the complete graph on $V(G)$. Then $G$ can be decomposed into $\Delta(G)/2$ cycles which are consistent with~$M$.
\end{theorem}

Clearly, if an Eulerian graph $G$ has a decomposition into $\Delta(G)/2$ cycles, then all vertices of maximum degree must be contained in every cycle of the decomposition. Our proof strategy for Theorem~\ref{thm:Eulerian} is thus to inductively take out cycles that contain all vertices of maximum degree. If we can make sure that the minimum degree decreases at a significantly slower rate than the maximum degree, we will eventually obtain a dense regular graph. We will ensure that this graph is still a robust outexpander so that we can apply Theorem~\ref{thm:hamdecrobexp} to decompose it into Hamilton cycles.

We will deduce the undirected version of Theorem~\ref{thm:Eulerian} by means of Lemma~\ref{lem:orient}. (Note that our proof of Theorem~\ref{thm:Eulerian} also works directly for the undirected case, by using Theorem~\ref{thm:hamdecrobexpundirected} instead of Theorem~\ref{thm:hamdecrobexp}.)

\lateproof{Theorem~\ref{thm:Eulerian}}
Let $\tau:=\tau(p/3)$ be as defined in Theorem~\ref{thm:hamdecrobexp}. We may assume $0<1/n_0\ll\eta,\eps\ll\tau,p<1$. By Proposition~\ref{prop:minsemideg}, $\delta^0(G)\geq pn/2$.
%\COMMENT{Consider a partition of $V(G)$ into sets $S,T$ such that $|S|=\eps n$ and $|T|=(1-\eps)n$. We have $e_G(S,T)\geq(p-\eps)|S||T|$. Moreover, since each vertex in $S$ sends at most $\Delta(G)/2$ edges to $T$, we have $e_G(S,T)\leq|S|\Delta(G)/2$. Thus, $\Delta(G)\geq2(p-\eps)(1-\eps)n$, implying $\delta^0(G)=\delta(G)/2\geq(p-\eps)(1-\eps)n-\eta n/2\geq pn/2$.}

Apply Lemma~\ref{lem:auxsets} to obtain a partition $S_1,S_2$ of $V(G)$ with $|S_1|,|S_2|\geq n/3$ such that
\begin{align}
d_{S_1}^+(x),d_{S_1}^-(x),d_{S_2}^+(x),d_{S_2}^-(x)\geq pn/12 \label{eqn:degrees}
\end{align}
for all $x\in V(G)$ and such that neither $S_1$ nor $S_2$ separates $M$. 

We will inductively construct spanning subgraphs $G_1\supseteq G_2\supseteq\dots\supseteq G_k$ of $G$, sets of vertices $Z_1\In Z_2\In\cdots\In Z_{k-1}$ and edge-disjoint cycles $C_1,\dots,C_{k-1}$ such that the following holds:
\begin{itemize}
\item[(P1)] $Z_i$ is the set of vertices of maximum degree in $G_i$,
\item[(P2)] $C_i$ is a cycle in $G_i$ that contains all vertices of $Z_i$,
\item[(P3)] if $i\geq2$, then $G_i$ is obtained from $G_{i-1}$ by removing the edges of $C_{i-1}$,
\item[(P4)] $\Delta(G_i)=\Delta(G)-2(i-1)$,
\item[(P5)] if $i\geq2$ and $d_{G_i}(v)<d_{G_{i-1}}(v)<\Delta(G_{i-1})$, then $v\in S_1$ if $i$ is even and $v\in S_2$ if $i$ is odd,
\item[(P6)] $G_k$ is regular,
\item[(P7)] $k\leq2\eta n$.
\end{itemize}
Let $G_1:=G$. 
Assume that for some $1\leq i\leq2\eta n$, we have already defined $G_1\supseteq\dots\supseteq G_i$, $Z_1\In\cdots\In Z_{i-1}$ and $C_1,\dots,C_{i-1}$ such that (P1)--(P5) are satisfied. If $G_i$ is regular, then stop. Otherwise, let $Z_i$ be the set of vertices of maximum degree in $G_i$. Since $G_{i-1}$ is Eulerian, we know that $Z_i\supseteq Z_{i-1}$.

We claim that $G_i[S_1\cup Z_i]$ and $G_i[S_2\cup Z_i]$ are Hamiltonian. Since $|S_1\cup Z_i|\geq n/3$, Proposition~\ref{prop:changes}(iii) implies that $G[S_1\cup Z_i]$ is lower-$(p,3\eps)$-regular. Let $H$ be the spanning subgraph of $G[S_1\cup Z_i]$ that contains the edges from $C_1,\dots,C_{i-1}$. Thus, $G[S_1\cup Z_i]\sm H=G_i[S_1\cup Z_i]$. As $\Delta(H)\leq2(i-1)\leq4\eta n\leq12\eta|S_1\cup Z_i|$, Proposition~\ref{prop:changes}(ii) implies that $G_i[S_1\cup Z_i]$ is lower-$(p,2\sqrt{12\eta})$-regular.\COMMENT{(ii) with $3\eps$ and $\eps'=\max\Set{6\eps,2\sqrt{12\eta}}$} Moreover, $\delta^0(G_i[S_1\cup Z_i])\geq pn/12-\Delta(H)\geq p|S_1\cup Z_i|/13$ by (\ref{eqn:degrees}). Thus, by Lemma~\ref{lem:hamcycle}(i), $G_i[S_1\cup Z_i]$ has a Hamilton cycle. The same applies to $S_2$.

Now, if $i$ is odd, then let $C_i$ be a Hamilton cycle in $G_i[S_1\cup Z_i]$, and if $i$ is even, then let $C_i$ be a Hamilton cycle in $G_i[S_2\cup Z_i]$. Set $G_{i+1}:=G_i\sm E(C_i)$.

By construction, (P1)--(P5) hold. 
Let us now show that the procedure will terminate after $k-1\leq2\eta n-1$ steps with an $r$-regular digraph $G_k$ for some $r\in \bN$. Indeed, first note that (P3) and (P5) together imply that
\begin{equation} \label{eqn:delta}
\delta(G_{i+1})\geq\delta(G)-(i+1).
\end{equation}
Thus, together with (P4), it follows that $$0\leq\Delta(G_{i+1})-\delta(G_{i+1})\leq\Delta(G)-2i-\delta(G)+i+1\leq\eta n-i+1$$ and so $i+1\leq2\eta n$.

(\ref{eqn:delta}) and (P7) imply that $\delta(G_k)\geq\delta(G)-k\geq pn-2\eta n$. Thus, $r\geq(p/2-\eta)n\geq pn/3$.
Since $E(G\sm G_k)=E(C_1)\cup\dots\cup E(C_{k-1})$, we have $\Delta(G\sm G_k)\leq2(k-1)\leq4\eta n$. Thus, $G_k$ is lower-$(p,4\sqrt{\eta})$-regular by Proposition~\ref{prop:changes}(ii) and therefore a robust $(4\sqrt{\eta},\tau)$-outexpander by Proposition~\ref{prop:REG2EXP}. Finally, we can apply Theorem~\ref{thm:hamdecrobexp} to decompose $G_k$ into Hamilton cycles. Let $\cC$ be the set consisting of all these Hamilton cycles together with the cycles $C_1,\dots,C_{k-1}$. We will show that $\cC$ is the required cycle decomposition of $G$.

Let $x$ be a vertex of maximum degree in $G$. Thus, $x\in Z_i$ for all $i<k$ and so $x$ is contained in every cycle in $\cC$. Hence, $|\cC|=d_G^+(x)=\Delta(G)/2$, as desired. By the choice of $S_1$ and $S_2$, all cycles in $\cC$ will be consistent with $M$. To see this, let $xy\in M$ and assume that $d_G(x)\leq d_G(y)$. Suppose that there is a cycle in $\cC$ which contains $x$ but not $y$. Let $i^\ast:=\min\set{i}{C_i\text{ contains }x\text{ but not }y}$. Since $V(C_{i^\ast})=Z_{i^\ast}\cup S_j$ for some $j\in\Set{1,2}$, and since neither $S_1$ nor $S_2$ separates $M$, it follows that $x\in Z_{i^\ast},y\notin Z_{i^\ast}$. In particular, $x$ has higher degree in $G_{i^\ast}$ than $y$, which is a contradiction, since by definition of $i^\ast$, any $C_j$ with $j<i^\ast$ which contains $x$ must also contain $y$. 
\endproof

Together with Lemma~\ref{lem:orient}, we obtain the following undirected version of Theorem~\ref{thm:Eulerian}, which immediately implies Theorem~\ref{thm:Eulerian plain}.

\begin{cor} \label{cor:Eulerian undirected}
Let $0<1/n_0\ll\eta,\eps\ll p<1$. Suppose $G$ is a lower-$(p,\eps)$-regular graph on $n\geq n_0$ vertices. Moreover, assume that $\Delta(G)-\delta(G)\leq\eta n$ and that $G$ is Eulerian. Let $M$ be any matching in the complete graph on $V(G)$. Then $G$ can be decomposed into $\Delta(G)/2$ cycles which are consistent with $M$.
\end{cor}

\section{Path decompositions} \label{sec:paths}

In this section, we will prove our quasirandom result concerning path and linear forest decompositions and then easily deduce Theorem~\ref{thm:gnp}. As mentioned earlier, any path decomposition of a graph $G$ must contain at least $\max\Set{odd(G)/2,\lceil\Delta(G)/2\rceil}$ paths, and $la(G)\geq\lceil\Delta(G)/2\rceil$. However, there are some simple cases in which these bounds cannot be attained. For example, suppose $G$ is $r$-regular and $r$ is even. Then it is impossible to decompose $G$ into $r/2$ paths or $r/2$ linear forests. This motivates the additional assumption in the following theorem that $G$ has a unique vertex of maximum degree. This additional assumption is stronger than what we would really need to prove Theorem~\ref{thm:pathforests}, but still enables us to deduce Theorem~\ref{thm:gnp}(ii)--(iii).

\begin{theorem} \label{thm:pathforests}
Let $0<1/n_0\ll\eta,\eps\ll p<1$. Suppose $G$ is a lower-$(p,\eps)$-regular graph on $n\geq n_0$ vertices such that $\Delta(G)-\delta(G)\leq\eta n$. Then the following hold.
\begin{itemize}
\item[(i)] $G$ can be decomposed into $\max\Set{odd(G)/2,\lceil(\Delta(G)+1)/2\rceil}$ paths. If $G$ has a unique vertex of maximum degree, then $G$ can be decomposed into $\max\Set{odd(G)/2,\lceil\Delta(G)/2\rceil}$ paths.
\item[(ii)] $G$ can be decomposed into $\lceil(\Delta(G)+1)/2\rceil$ linear forests. If $G$ has a unique vertex of maximum degree, then $G$ can be decomposed into $\lceil\Delta(G)/2\rceil$ linear forests.
\end{itemize}
\end{theorem}

To prove Theorem~\ref{thm:pathforests}, we will construct a suitable Eulerian auxiliary graph $G'$ to which we can apply Theorem~\ref{thm:Eulerian} or Corollary~\ref{cor:Eulerian undirected}. The resulting cycle decomposition of $G'$ will be transformed into a path decomposition of $G$.

\proof
Let $D=\Set{a_1,b_1,a_2,b_2,\dots,a_s,b_s}$ be the set of odd degree vertices in $G$. If $a_ib_i\notin E(G)$, we still refer to the pair $\Set{a_i,b_i}$ as the edge $a_ib_i$, thereby meaning the respective edge in the complement of $G$.

\noindent\textbf{Case 1:} $odd(G)\geq\Delta(G)$

Let $G'$ be the graph obtained from $G$ in the following way. Add a new vertex $w$ and choose $\lceil\Delta(G)/2\rceil$ of the pairs $\Set{a_i,b_i}$. Connect $w$ to all the vertices in these pairs. For the remaining pairs, let $E^\star$ be the set of edges $a_ib_i\in E(G)$ and let $E^\circ$ be the set of edges $a_ib_i$ which are in the complement of $G$. Remove $E^\star$ from the graph and add $E^\circ$. Hence, $G'$ is Eulerian, $\Delta(G')-\delta(G')\leq\Delta(G)-\delta(G)+2$, and by Proposition~\ref{prop:changes}(i) and (ii), $G'$ is lower-$(p,4\eps)$-regular. We thus can apply Corollary~\ref{cor:Eulerian undirected} to  $G'$ to find a decomposition of $G'$ into $\Delta(G')/2$ cycles which are consistent with $E^\star$. Since $d_{G'}(w)=2\lceil\Delta(G)/2\rceil=\Delta(G')$, $w$ is contained in each of these cycles. This naturally gives a decomposition $\cP$ of $(G\sm E^\star)\cup E^\circ$ into $\lceil\Delta(G)/2\rceil$ paths.

We now use $\cP$ to prove (i). 
Remove the edges of $E^\circ$ one by one from the paths in $\cP$. Every time when an edge is removed, the number of paths increases by one. Finally, for every edge in $E^\star$, simply take this edge as a path of length one. Thus, we have a decomposition of $G$ into $\lceil\Delta(G)/2\rceil+|E^\circ|+|E^\star|=odd(G)/2$ paths.

We next use $\cP$ to prove (ii).
When we are looking for linear forests, we can remove the edges of $E^\circ$ from the paths in $\cP$ and thus obtain a collection $\cF$ of $\lceil\Delta(G)/2\rceil$ linear forests in $G-E^\star$. We now distribute the edges of $E^\star$ among these linear forests. To this end, consider any $a_ib_i\in E^\star$. We may assume that $d_G(a_i)\leq d_G(b_i)$. Since $a_ib_i$ was deleted from $G$, $d_{G'}(b_i)<\Delta(G)$. Moreover, $b_i$ is not a leaf in any of the linear forests in $\cF$. Therefore, there is a linear forest $F\in\cF$ which does not contain $b_i$. Since the cycle decomposition of $G'$ was chosen to be consistent with $E^\star$, $F$ does not contain $a_i$ either. Hence, we can add $a_ib_i$ to $F$ to obtain a new linear forest. We proceed analogously with the other edges in $E^\star$ to obtain a decomposition of $G$ into $\lceil\Delta(G)/2\rceil$ linear forests.

\noindent\textbf{Case 2:} $odd(G)<\Delta(G)$

Let $E^\star$ contain all edges $a_ib_i\in E(G)$ and let $E^\circ$ contain all edges $a_ib_i$ which are in the complement of $G$. Let $G':=(G\sm E^\star)\cup E^\circ$. Hence, $G'$ is Eulerian and lower-$(p,2\eps)$-regular by Proposition~\ref{prop:changes}(ii). By Lemma~\ref{lem:orient}, there exists an orientation $G^{orient}$ of $G'$ which is Eulerian and lower-$(p/4,2\eps)$-regular. Let $G''$ be the digraph obtained from $G^{orient}$ in the following way. Add a new vertex $w$ and connect it to all vertices of $D$ by one edge as follows. 
If $a_ib_i\in E(G^{orient})$, remove it and add $a_iw$ and $wb_i$. If $b_ia_i\in E(G^{orient})$, remove it and add $b_iw$ and $wa_i$. Else, add the edges $a_ib_i$, $b_iw$ and $wa_i$. Note that the graph obtained from $G''-w$ by ignoring the orientation of its edges is precisely $G$. 
\COMMENT{Note that we cannot add $w$ first and connect it to all odd degree vertices, then apply the orientation lemma, as the degree of $w$ might be too small.}

\noindent\emph{Case 2.1:} $G$ has a unique vertex of maximum degree and $\Delta(G)$ is even.

Let $k:=(\Delta(G)-odd(G))/2$. Since $G$ has a unique vertex of maximum degree, we can choose $k$ distinct vertices not contained in $D$ and not of maximum degree and connect $w$ to each of them by two edges which are oriented in opposite directions.\COMMENT{There are $n-odd(G)-1$ suitable vertices and $k\leq(n-odd(G)-1)/2$.} The digraph $G'''$ obtained from $G''$ in this way is Eulerian and $d_{G'''}(w)=\Delta(G''')=\Delta(G)$. Note that Proposition~\ref{prop:changes} implies that $G'''$ is lower-$(p/4,8\eps)$-regular.\COMMENT{$G^{orient}$ is lower-$(p/4,2\eps)$-regular, $G'''$ obtained from $G^{orient}$ by (i) adding $w$ and some edges at $w$ and (ii)  deleting $E^\circ$.} Decompose $G'''$ into $\Delta(G)/2$ cycles by means of Theorem~\ref{thm:Eulerian}. Recall that the graph obtained from $G''-w=G'''-w$ by ignoring the orientation of its edges is precisely $G$. Together with the fact that $d_{G'''}(w)=\Delta(G''')$, this implies that the cycle decomposition of $G'''$ yields a decomposition of $G$ into $\Delta(G)/2$ paths (and thus also a decomposition into $\Delta(G)/2$ linear forests).

\noindent\emph{Case 2.2:} $G$ does not have a unique vertex of maximum degree or $\Delta(G)$ is odd.

Let $k':=\lceil(\Delta(G)+1-odd(G))/2\rceil$. We pick $k'$ distinct vertices not contained in $D$ and connect $w$ to each of them by two edges which are oriented in opposite directions.\COMMENT{There are $n-odd(G)$ suitable vertices and $k'\leq\lceil(n-odd(G))/2\rceil$.} Let $G'''$ be the digraph obtained from $G''$ in this way. Then $d_{G'''}(w)=odd(G)+2k'=\Delta(G''')=2\lceil(\Delta(G)+1)/2\rceil$. So we can argue similarly as in Case~2.1 to obtain a decomposition of $G$ into $\lceil(\Delta(G)+1)/2\rceil$ paths (and thus also a decomposition into $\lceil(\Delta(G)+1)/2\rceil$ linear forests).
\endproof

\lateproof{Theorem~\ref{thm:gnp}}
Let $G\sim\gnp$. Let $D$ denote the set of odd degree vertices in $G$. Recall that \whp $G$ has the properties stated in Lemma~\ref{lem:gnpbasics}. Moreover, Lemma~\ref{lem:odd matching} implies that \whp $G[D]$ has a perfect matching $M$.

To prove (i), we can apply Theorem~\ref{thm:Eulerian plain} to $G-M$. Since $\Delta(G-M)=2\lfloor\Delta(G)/2\rfloor$, $G-M$ can be decomposed into $\lfloor\Delta(G)/2\rfloor$ cycles.

Parts (ii) and (iii) follow immediately from Theorem~\ref{thm:pathforests}.
\endproof

\section{Linear arboricity} \label{sec:forests}

Recall that Theorem~\ref{thm:gnp}(iii) verifies the linear arboricity conjecture for $\gnp$ when $p$ is constant. In this section we additionally verify this conjecture for the following situations: (i) for $\gnp$ where $p\rightarrow0$ but $p$ is not too small, and (ii) for large and sufficiently dense regular graphs. We obtain these results as corollaries of (i) an optimal Hamilton cover result in \cite{HKLO}, and (ii) the recent proof of the Hamilton decomposition conjecture in \cite{CBKLOT}.

\begin{theorem}[Hefetz, K{\"u}hn, Lapinskas, Osthus \cite{HKLO}] \label{thm:hamcover}
Let $\frac{\log^{117}n}{n}\leq p\leq 1-n^{-1/8}$ and $G\sim\gnp$. Then \whp the edges of $G$ can be covered by $\lceil\Delta(G)/2\rceil$ Hamilton cycles.
\end{theorem}

\begin{cor} \label{cor:forests random}
Let $\frac{\log^{117}n}{n}\leq p=o(1)$ and $G\sim\gnp$. Then \whp $G$ can be decomposed into $\lceil\Delta(G)/2\rceil$ linear forests.
\end{cor}

\proof
We note that for the given range of $p$, Lemma~\ref{lem:gnpbasics}(iii) still holds, which again can be seen by applying Theorem 3.15 from \cite{B}. Let $G'$ be obtained from $G$ by adding a new vertex $w$ and adding each potential edge at $w$ independently with probability $p$. Hence, $G'\sim G_{n+1,p'}$, where (as a function of the order $n+1$ of the graph) $p'(n+1)=p(n)$.
By Theorem~\ref{thm:hamcover}, \whp the edges of $G'$ can be covered by $\lceil\Delta(G')/2\rceil$ Hamilton cycles. Clearly, this Hamilton cycle cover of $G'$ yields a decomposition of $G$ into $\lceil\Delta(G')/2\rceil$ linear forests. It remains to check that \whp $\Delta(G')=\Delta(G)$. \Whp $G$ has a unique vertex $v_{max}$ of maximum degree. Since $p=o(1)$, it follows that \whp $d_{G'}(v_{max})=d_G(v_{max})$. Moreover, \whp $G'$ has a unique vertex of maximum degree. The probability that $w$ is this very vertex is $1/(n+1)=o(1)$. Hence, \whp $\Delta(G')=d_{G'}(v_{max})=\Delta(G)$.
\endproof

With more work, the above argument can also be extended to the case when $p$ is constant, which would give an alternative proof of Theorem~\ref{thm:gnp}(iii).

We now use the result of \cite{CBKLOT} mentioned earlier to deduce that the linear arboricity conjecture holds for large and sufficiently dense regular graphs.

\begin{theorem}[Csaba, K{\"u}hn, Lo, Osthus, Treglown \cite{CBKLOT}] \label{thm:hamdec}
There exists an $n_0\in\bN$ such that whenever $n\geq n_0$, $d\geq\lfloor n/2\rfloor$ and $G$ is a $d$-regular graph on $n$ vertices, then $G$ has a decomposition into Hamilton cycles and at most one perfect matching.
\end{theorem}

\begin{cor} \label{cor:forests large}
For sufficiently large $n$, every $d$-regular graph $G$ on $n$ vertices with $d\geq\lfloor(n-1)/2\rfloor$ has a decomposition into $\lceil(d+1)/2\rceil$ linear forests.
\end{cor}

\begin{proof}
In our argument below we will use the fact that any graph $H$ with minimum degree at least $r$ has a matching that covers at least $r$ vertices.\COMMENT{To see this, consider a longest path in $H$. By the minimum degree condition, such a path contains at least $r$ edges. By picking every other edge, we get a matching of size at least $\lceil r/2\rceil$.}

\noindent\textbf{Case 1:} $d$ is odd.

Then $n$ must be even. Clearly, the complement $\overline{G}$ of $G$ is $(n-1-d)$-regular. Hence, there is a matching $M$ in $\overline{G}$ that covers $n-(d+1)$ vertices. Let $G'$ be the graph obtained from $G$ by adding the edges of $M$ and joining a new vertex $w$ to the $d+1$ vertices not covered by $M$. Hence, $G'$ is $(d+1)$-regular and $d+1\geq\lfloor|G'|/2\rfloor$. By Theorem~\ref{thm:hamdec} and because $d+1$ is even, $G'$ has a decomposition into $(d+1)/2$ Hamilton cycles. By removing $w$ and the edges of $M$ from each of these Hamilton cycles, we obtain a decomposition of $G$ into $(d+1)/2$ linear forests.

\noindent\textbf{Case 2:} $d$ is even and $n$ is even.

Let $M$ be a matching in $\overline{G}$ that covers $n-(d+2)$ vertices. Since $\overline{G}\sm M$ has minimum degree at least $n-(d+2)$, $\overline{G}\sm M$ contains another matching $M'$ that covers $n-(d+2)$ vertices. Let $G'$ be the graph obtained from $G$ by adding the edges of $M$ and $M'$ and two new vertices $w,w'$ such that $w$ is joined to the $d+2$ vertices in $G$ not covered by $M$ and $w'$ is joined to the $d+2$ vertices in $G$ not covered by $M'$. Thus, $G'$ is $(d+2)$-regular and $d+2\geq\lfloor|G'|/2\rfloor$. By Theorem~\ref{thm:hamdec} and because $d+2$ is even, $G'$ has a decomposition into $(d+2)/2$ Hamilton cycles. By removing $w$, $w'$ and the edges of $M\cup M'$ from each of these Hamilton cycles, we obtain a decomposition of $G$ into $(d+2)/2=\lceil(d+1)/2\rceil$ linear forests.

\noindent\textbf{Case 3:} $d$ is even and $n$ is odd.

Let $M$ be a matching in $\overline{G}$ that covers $n-(d+1)$ vertices. Let $G'$ be the graph obtained from $G$ by adding the edges of $M$ and joining a new vertex $w$ to the $d+1$ vertices not covered by $M$. Then $G'$ is $(d+1)$-regular and $d+1\geq\lfloor(|G'|-1)/2\rfloor$. By Case 1, $G'$ can be decomposed into $(d+2)/2=\lceil(d+1)/2\rceil$ linear forests. As $G$ is a subgraph of $G'$, $G$ can be decomposed into the same number of linear forests.
\end{proof}

\section{Edge colourings} \label{sec:colouring}

In this section we build on Theorem~\ref{thm:hamdecrobexpundirected} to prove the overfull subgraph conjecture for dense quasirandom graphs of even order. We need the following well-known result on multigraphic degree sequences.

\begin{theorem}[Hakimi \cite{Ha}] \label{thm:multigraphic}
Let $0\leq d_n\leq\dots\leq d_1$ be integers. Then there exists a multigraph $G$ on vertices $x_1,\dots,x_n$ such that $d_G(x_i)=d_i$ for all $i$ if and only if $\sum_{i=1}^n d_i$ is even and $\sum_{i>1}d_i\geq d_1$.
\end{theorem}

Though it is not explicitly stated in \cite{Ha}, the inductive proof yields a polynomial time algorithm which finds an appropriate multigraph if it exists.\COMMENT{The algorithm for graphic sequences is known as the Havel-Hakimi-algorithm, easy to see running time is $\O(n^2\log n)$ and can be improved to $\O(n^2)$.} Our strategy in proving Theorem~\ref{thm:colouring} is similar to that of Theorem~\ref{thm:Eulerian}, yet different in detail. Roughly speaking, instead of inductively removing cycles, we aim to remove paths in order to make our graph regular and then apply Theorem~\ref{thm:hamdecrobexpundirected} to decompose the regular remainder into Hamilton cycles. We can then simply colour each path with two colours and, since our graph has even order, each Hamilton cycle with two colours. An auxiliary multigraph obtained from Theorem~\ref{thm:multigraphic} will tell us which vertices to choose as the endvertices of the paths to be removed. In order to actually find these paths, we observe that lower-$(p,\eps)$-regular graphs contain `spanning linkages' for arbitrary pairs of vertices.

\begin{lemma} \label{lem:pathcontrol}
Let $0<1/n_0\ll\eps\ll\alpha,p<1$.
Let $G$ be a lower-$(p,\eps)$-regular graph on $n\geq n_0$ vertices such that $\delta(G)\geq\alpha n$. Moreover, let $M=\Set{a_1b_1,\dots,a_mb_m}$ be a matching in the complete graph on $V(G)$ of size at most $\alpha n/5$. Then there exist vertex-disjoint paths $P_1,\dots,P_m$ in $G$ such that $\bigcup V(P_i)=V(G)$ and $P_i$ joins $a_i$ to $b_i$, and these paths can be found in polynomial time.
\end{lemma}

\proof
Our paths $P_1,\dots,P_{m-1}$ will have length at most $3$ and $P_m$ will contain all the remaining vertices. Suppose that for some $1\leq i\leq m$, we have already defined vertex-disjoint paths $P_1,\dots,P_{i-1}$ in $G$ of length at most $3$ and such that $P_j$ joins $a_j$ to $b_j$ and is disjoint from $\Set{a_i,b_i,\dots,a_m,b_m}$. Let $X_i:=(V(M)\sm\Set{a_i,b_i})\cup\bigcup_{j=1}^{i-1} V(P_j)$. Let $G_i:=G-X_i$. Since $|X_i|\leq2m+2(i-1)\leq4m\leq4\alpha n/5$, we have $\delta(G_i)\geq\alpha n/5$. As $|G_i|\geq n/5$, $G_i$ is lower-$(p,5\eps)$-regular by Proposition~\ref{prop:changes}(iii). If $i<m$, we seek a path of length at most $3$. Suppose there is no path of length at most two in $G_i$ joining $a_i$ to $b_i$. Then $N_{G_i}(a_i)$ and $N_{G_i}(b_i)$ are disjoint and both have size at least $\alpha n/5\geq5\eps n$. Since $G_i$ is lower-$(p,5\eps)$-regular, there exists an edge in $G_i$ between $N_{G_i}(a_i)$ and $N_{G_i}(b_i)$ which gives us a path $P_i$ of length $3$ joining $a_i$ and $b_i$. Finally, if $i=m$, then Lemma~\ref{lem:hamcycle}(iii) tells us that we can find a Hamilton path $P_m$ in $G_m$ from $a_m$ to $b_m$ (in polynomial time).
\endproof

For a graph $G$ and $x\in V(G)$, let $\mbox{def}_G(x):=\Delta(G)-d_G(x)$ denote the \emph{deficiency of $x$ in $G$}.

\lateproof{Theorem~\ref{thm:colouring}}
Choose constants $\eps,\eta$ and $n_0\in\bN$ such that $0<1/n_0\ll\eps\leq\eta\ll p$. Let $G$ be any lower-$(p,\eps)$-regular graph on $n\geq n_0$ vertices with $\Delta(G)-\delta(G)\leq\eta n$, where $n$ is even. Let $\alpha:=p/2$. It is then easy to see that $\delta(G)\geq \alpha n$.\COMMENT{Consider a partition of $V(G)$ into sets $S,T$ such that $|S|=\eps n$ and $|T|=(1-\eps)n$. We have $e_G(S,T)\geq(p-\eps)|S||T|$. Moreover, since each vertex in $S$ sends at most $\Delta(G)$ edges to $T$, we have $e_G(S,T)\leq|S|\Delta(G)$. Thus, $\Delta(G)\geq(p-\eps)(1-\eps)n$, implying $\delta(G)\geq(p-\eps)(1-\eps)n-\eta n\geq pn/2$.}

If $G$ contains any overfull subgraph, then $\chi'(G)>\Delta(G)$. So let us assume that $G$ contains no overfull subgraph.
Label the vertices of $G$ such that $d_G(x_1)\leq\dots\leq d_G(x_n)$ and let $\mbox{def}_i:=\mbox{def}_G(x_i)$.

We claim that $\mbox{def}_1\leq\sum_{i>1}\mbox{def}_i$. Suppose not, and let $H:=G-x_1$. Then $e(H)=e(G)-d_G(x_1)$ and so
\begin{align*}
2e(H)&=\sum_{i=1}^n d_G(x_i)-2d_G(x_1)=\sum_{i>1} d_G(x_i)-d_G(x_1) \\
     &=\sum_{i>1} (\Delta(G)-\mbox{def}_i)-\Delta(G)+\mbox{def}_1>(n-2)\Delta(G)=(|H|-1)\Delta(G).
\end{align*}
Since $|H|$ is odd, $H$ is overfull, a contradiction. Hence, we have $\mbox{def}_1\leq\sum_{i>1}\mbox{def}_i$. Moreover, since $n$ is even, the sum of all deficiencies is even. So by Theorem~\ref{thm:multigraphic}, there exists a multigraph $A$ on $V(G)$ such that $d_A(x_i)=\mbox{def}_i$. We will use this auxiliary multigraph $A$ to find a dense spanning regular subgraph $G_k$ of $G$. Observe that $\Delta(A)=\mbox{def}_1=\Delta(G)-\delta(G)\leq\eta n$. Thus, $\chi'(A)\leq2\eta n$. Hence, we can (greedily) partition $E(A)$ into $k\leq 6\eta n/\alpha$ matchings $M_1,\dots,M_k$ of size at most $\alpha n/6$.\COMMENT{If $M$ is a matching in $A$, then we can break it into at most $3/\alpha$ submatchings of size at most $\alpha n/6$.} We will now inductively take out linear forests from $G$ by applying Lemma~\ref{lem:pathcontrol} with $M_1,\dots,M_k$. More precisely, we define spanning subgraphs $G_0,\dots,G_k$ of $G$ and edge-disjoint linear forests $F_1,\dots,F_k$ such that
\begin{itemize}
\item[(a)] $G_0:=G$ and $G_i=G_{i-1}\sm F_i$ for $i\geq1$, 
\item[(b)] $F_i$ is a spanning linear forest in $G_{i-1}$ whose leaves are precisely the vertices in $M_i$.
\end{itemize}
Let $G_0:=G$ and suppose that for some $1\leq i\leq k$ we have already defined $G_0,\dots,G_{i-1}$ and $F_1,\dots,F_{i-1}$. Since $\Delta(F_1\cup\dots\cup F_{i-1})\leq2(i-1)\leq12\eta n/\alpha$, it follows that $\delta(G_{i-1})\geq(\alpha-12\eta/\alpha)n\geq5\alpha n/6$. Moreover, let $\eps':=2\sqrt{12\eta/\alpha}$. Then, by Proposition~\ref{prop:changes}(ii), $G_{i-1}$ is lower-$(p,\eps')$-regular. Hence, since $M_i$ has size at most $\alpha n/6$, we can apply Lemma~\ref{lem:pathcontrol} to $G_{i-1}$ and $M_i$ to obtain a spanning linear forest $F_i$ in $G_{i-1}$ whose leaves are precisely the vertices in $M_i$. Set $G_i:=G_{i-1}\sm F_i$.

We claim that $G_k$ is regular. Consider any vertex $x\in V(G)$. Then, $d_{G_k}(x)=d_G(x)-\sum_{i=1}^k d_{F_i}(x)$. For every $1\leq i\leq k$, $d_{F_i}(x)=1$ if $x$ is an endvertex of some edge of $M_i$, and $d_{F_i}(x)=2$ otherwise. Since $M_1,\dots,M_k$ partition $E(A)$, we conclude that $\sum_{i=1}^k d_{F_i}(x)=2k-d_A(x)=2k-\mbox{def}_G(x)$. Thus, $d_{G_k}(x)=d_G(x)-2k+\mbox{def}_G(x)=\Delta(G)-2k$.

Let $r:=\Delta(G)-2k$. We have shown that $G_k$ is $r$-regular. Moreover, $r\geq5\alpha n/6$ and $G_k$ is lower-$(p,\eps')$-regular. If $r$ is even, we can decompose $G_k$ into Hamilton cycles by means of Theorem~\ref{thm:hamdecrobexpundirected} and Proposition~\ref{prop:REG2EXP}.
If $r$ is odd, we can use Lemma~\ref{lem:hamcycle}(ii) to find a perfect matching $M^\ast$ in $G_k$. Let $G_k^\ast:=G_k\sm M^\ast$, then $G_k^\ast$ is even-regular and we can proceed analogously.

In order to obtain an optimal edge colouring of $G$, we colour each $F_i$ with two colours, each Hamilton cycle with two colours, and if $r$ is odd, $M^\ast$ with one additional colour. Since $x_n$ is incident with an edge of every colour, we use $\Delta(G)$ colours.

Let us finally check that this yields a polynomial time algorithm. Given $G$, first check if $\mbox{def}_1\leq\sum_{i>1}\mbox{def}_i$. If not, then we know that $\chi'(G)=\Delta(G)+1$ and can use any polynomial time algorithm that attains Vizing's bound (see e.g. \cite{GM}). If the inequality holds, then we can construct a $\Delta(G)$-colouring by proceeding as above. Since Theorem~\ref{thm:multigraphic}, Lemma~\ref{lem:pathcontrol}, Lemma~\ref{lem:hamcycle}(ii) and Theorem~\ref{thm:hamdecrobexpundirected} give appropriate running time statements, this can be achieved in time polynomial in $n$.
\endproof

\vspace{1cm}

{\footnotesize \obeylines \parindent=0pt

Stefan Glock, Daniela K\"{u}hn, Deryk Osthus
School of Mathematics
University of Birmingham
Edgbaston
Birmingham
B15 2TT
UK
}
\begin{flushleft}
{\it{E-mail addresses}:
\tt{\{sxg426,d.kuhn,d.osthus\}@bham.ac.uk} }
\end{flushleft}


\begin{thebibliography}{10}

\bibitem{AEH} J.~Akiyama, G.~Exoo and F.~Harary, Covering and packing in graphs. III: Cyclic and acyclic invariants, \emph{Math. Slovaca}~{\bf 30} (1980), 405--417.

\bibitem{A} N.~Alon, The linear arboricity of graphs, \emph{Isr. J. Math.}~{\bf 62} (1988), 311--325.

\bibitem{B} B.~Bollob\'as, \emph{Random graphs}, 2nd ed., Cambridge Stud. Adv. Math. 73, Cambridge University
Press, 2001.

\bibitem{BHV} K.~Bongard, A.~Hoffmann and L.~Volkmann, Minimum degree conditions for the Overfull Conjecture for odd order graphs, \emph{Australasian Journal of Combinatorics}~{\bf 28} (2003), 121--129.

\bibitem{CH} A.G. Chetwynd and A.J.W. Hilton, Regular graphs of high degree are $1$-factorizable, \emph{Proc. London Math. Soc.}~{\bf 50} (1985), 193--206.

\bibitem{CFS}  D.~Conlon, J.~Fox and B.~Sudakov, Cycle packing, \emph{Random Structures and Algorithms}~{\bf 45} (2014), 608--626.

\bibitem{CBKLOT} B.~Csaba, D.~K{\"u}hn, A.~Lo, D.~Osthus and A.~Treglown, Proof of the $1$-factorization and Hamilton decomposition conjectures, \emph{Memoirs of the American Mathematical Society}, to appear.

\bibitem{DK} N.~Dean and M.~Kouider, Gallai's conjecture for disconnected graphs, \emph{Discrete Math.}~{\bf 213} (2000), 43--54.

\bibitem{E} P.~Erd\H{o}s, On some of my conjectures in number theory and combinatorics, Proceedings of the fourteenth Southeastern conference on combinatorics, graph theory and computing (Boca Raton, Fla., 1983), \emph{Congr. Numer.}~{\bf 39} (1983), 3--19.

\bibitem{EGP} P.~Erd\H{o}s, A.W.~Goodman and L.~P\'osa, The representation of a graph by set intersections, \emph{Canad. J. Math.}~{\bf 18} (1966), 106--112.

\bibitem{GPS} P.~Gao, X.~P\'erez-Gim\'enez and C.M.~Sato, Arboricity and spanning-tree packing in random graphs with an application to load balancing, Proceedings of the Twenty-Fifth Annual ACM-SIAM Symposium on Discrete Algorithms (SODA '14), 317--326.

\bibitem{GM} D.~Gries and J.~Misra, A constructive proof of Vizing's Theorem, \emph{Information Processing Letters}~{\bf 41} (1992), 131--133.

\bibitem{Ha} S.L. Hakimi, On realizability of a set of integers as degrees of the vertices of a linear graph. I, \emph{J. Soc. Indust. Appl. Math.}~{\bf 10} (1962), 496--506.

\bibitem{HKLO} D.~Hefetz, D.~K{\"u}hn, J.~Lapinskas and D.~Osthus, Optimal covers with Hamilton cycles in random graphs, \emph{Combinatorica}~{\bf 34} (2014), 573--596.

\bibitem{Ho} I.~Holyer, The NP-completeness of edge-colouring, \emph{SIAM J. Comput.}~{\bf 10} (1981), 718--720.

\bibitem{JLR} S.~Janson, T.~{\L}uczak and A.~Ruci{\'n}ski, \emph{Random Graphs}, Wiley, New York, 2000.

\bibitem{KKO} F.~Knox, D.~K{\"u}hn and D.~Osthus, Edge-disjoint Hamilton cycles in random graphs, \emph{Random Structures and Algorithms}~{\bf 46} (2015), 397--445. 

\bibitem{KKS}  D.~Kor{\'a}ndi, M.~Krivelevich and B.~Sudakov, Decomposing random graphs into few cycles and edges, \emph{Combinatorics, Probability and Computing}~{\bf 24} (2015), 857--872.

\bibitem{KS} M.~Krivelevich and W.~Samotij, Optimal packings of Hamilton cycles in sparse random graphs, \emph{SIAM J. Discrete Mathematics}~{\bf 26} (2012), 964--982.

\bibitem{KO:13}  D.~K{\"u}hn and D.~Osthus, Hamilton decompositions of regular expanders: A proof of Kelly's conjecture for large tournaments, \emph{Advances in Mathematics}~{\bf 237} (2013), 62--146.

\bibitem{KO:14}  D.~K{\"u}hn and D.~Osthus, Hamilton decompositions of regular expanders: applications, \emph{J. Combin. Theory~B}~{\bf 104} (2014), 1--27.

\bibitem{KOT}  D.~K{\"u}hn, D.~Osthus and A.~Treglown, Hamiltonian degree sequences in digraphs, \emph{J. Combin. Theory~B}~{\bf 100} (2010), 367--380.

\bibitem{L} L.~Lov\'asz, On covering of graphs, in: P.~Erd\H{o}s and G.~Katona (Eds.), \emph{Theory of Graphs}, Academic Press, New York, 1968, 231--236.

\bibitem{MR} C.~McDiarmid and B.~Reed, Linear arboricity of random regular graphs, \emph{Random Structures and Algorithms}~{\bf 1} (1990), 443--445.

\bibitem{N} T.~Niessen, How to find overfull subgraphs in graphs with large maximum degree, \emph{Discrete Applied Mathematics}~{\bf 51} (1994), 117--125.

\bibitem{P} M.~Plantholt, Overfull conjecture for graphs with high minimum degree, \emph{Journal of Graph Theory}~{\bf 47} (2004), 73--80.

\bibitem{SSTF} M.~Stiebitz, D.~Scheide, B.~Toft, and L.M.~Favrholdt, \emph{Graph Edge Coloring: Vizing's Theorem and Goldberg's Conjecture}, Wiley, Hoboken, 2012.

\bibitem{Y} L.~Yan, On path decompositions of graphs, PhD thesis, Arizona State University, 1998.

\end{thebibliography}
\end{document}